\documentclass[11pt,reqno]{amsproc}
\usepackage{amsmath, amsthm, amssymb}
\usepackage{times, esint,stackrel}
\usepackage{enumitem}
\usepackage{color}
\usepackage{enumitem}   
\usepackage{etoolbox}
\usepackage{mathrsfs}
\usepackage{framed}

\usepackage[top=0.8in, bottom=0.8in, left=0.8in, right=0.8in]{geometry}

\usepackage[dvipsnames]{xcolor}

\usepackage[colorlinks=true, pdfstartview=FitV, linkcolor=RoyalBlue,citecolor=ForestGreen, urlcolor=blue]{hyperref}

\makeatletter
\patchcmd\@thm
  {\let\thm@indent\indent}{\let\thm@indent\noindent}%
  {}{}
\makeatother

\usepackage{etoolbox}
\expandafter\patchcmd\csname\string\proof\endcsname
  {\normalparindent}{0pt }{}{}

\theoremstyle{plain}
\newtheorem{THEOREM}{Theorem}[section]

\newtheorem{theorem}[THEOREM]{Theorem}
\newtheorem{corollary}[THEOREM]{Corollary}

\newtheorem{proposition}[THEOREM]{Proposition}

\theoremstyle{definition}

\theoremstyle{remark}

\newtheorem{remark}[THEOREM]{Remark}

\newcommand{\be}{\begin{equation}}
\newcommand{\ee}{\end{equation}}
\newcommand{\bea}{\begin{eqnarray}}
\newcommand{\eea}{\end{eqnarray}}

\newcommand{\rmd}{{\rm d}}

\newcommand{\bq}{\begin{equation}}
\newcommand{\eq}{\end{equation}}
\newcommand{\bqa}{\begin{eqnarray*}}
\newcommand{\eqa}{\end{eqnarray*}}

\newcommand{\cE}{\mathcal{E}}

\newcommand{\p}{\partial}

\def\cl{ \underline{\rho} }

\def \nl {\mathrm{nl}}
\def \loc {\mathrm{loc}}

\def\cH{\mathcal{H}}
\def\cL{\mathcal{L}}

\def\dz{\partial_z}

\def\cL{{\mathcal L}}

\def\T{{\mathbb T}}

\def\beq{\begin{equation}}
\def\eeq{\end{equation}}

\def \a {\alpha}

\def \g {\gamma}
\def \d {\delta}
\def \e {\varepsilon}

\def \l {\lambda}

\def \s {\sigma}
\def \t {\tau}
\def \th {\theta}

\def \cD {\mathcal{D}}
\def \cE {\mathcal{E}}

\def \cH {\mathcal{H}}

\def \cL {\mathcal{L}}
\def \cM {\mathcal{M}}

\newcommand{\N}{\ensuremath{\mathbb{N}}}   
\newcommand{\Z}{\ensuremath{\mathbb{Z}}}   
\newcommand{\R}{\ensuremath{\mathbb{R}}}   

\def \p {\partial}

\def \ss {\subset}

\renewcommand{\geq}{\geqslant}
\renewcommand{\ge}{\geqslant}
\renewcommand{\leq}{\leqslant}
\renewcommand{\le}{\leqslant}

\def \dx  {\, \mbox{d}x}
\def \dt  {\, \mbox{d}t}

\def \dy  {\, \mbox{d}y}
\def \dz  {\, \mbox{d}z}

\def \ds  {\, \mbox{d}s}

\def \dth  {\, \mbox{d}\th}

\def \ddt  {\frac{\mbox{d\,\,}}{\mbox{d}t}}

\def \dD  {\mbox{D}}

\begin{document}

\title[Entropy Hierarchies]{Entropy Hierarchies for equations of\\ compressible fluids and self-organized dynamics}

\author{Peter Constantin}
\address{Department of Mathematics, Princeton University, Princeton, NJ 08544}
\email{const@math.princeton.edu}

\author{Theodore D. Drivas}
\address{Department of Mathematics, Princeton University, Princeton, NJ 08544}
\email{tdrivas@math.princeton.edu}

\author{Roman Shvydkoy}
\address{Department of Mathematics, Statistics and Computer Science, University of Illinois at Chicago, 60607}
\email{shvydkoy@uic.edu}


\date{today}


\thanks{\textbf{Acknowledgment.}  
 The research of PC was partially supported by NSF grant DMS-1713985.
Research of TD was partially supported by
NSF grant DMS-1703997.  Research of RS was supported in part by NSF
	grants DMS 1515705, DMS-1813351 and Simons Foundation. He thanks Princeton University for hospitality during the work on the paper.}

\begin{abstract}
	We develop a method of obtaining a hierarchy of new higher-order entropies in the context of compressible models with local and non-local diffusion, and isentropic pressure.  The local viscosity is allowed to degenerate as the density approaches vacuum.  The method provides a tool to propagate initial regularity of classical solutions provided no vacuum has formed and serves as an alternative to the classical energy method.  We obtain a series of global well-posedness results for state laws in previously uncovered cases  including  $p(\rho) = c_p \rho$.  As an application we prove global well-posedness of collective behavior models with pressure arising from agent-based Cucker-Smale system. 
\end{abstract}

\maketitle

\section{Introduction}

We consider a class of compressible fluid models in one space dimension with periodic boundary conditions
\begin{align}
	&\partial_t \rho + \partial_x (u \rho ) = 0, \label{eq:mass} \\
	&\partial_t (\rho u) + \partial_x (\rho u^2) =- \partial_x p(\rho) + \cD(u,\rho) +\rho f\label{eq:mom}, \\
	&(\rho,u)|_{t = 0} = (\rho_0, u_0). \label{eq:initv}
\end{align}
Here $f \in L^\infty(\T \times \R^+)$ is a bounded external force,   $\cD(u,\rho)$ is a diffusion operator and the pressure $p:=p(\rho)$
is a given function of density.  The central feature of dissipation operators $\cD$ considered here is the existence of another quantity, denoted $Q$, which allows one to express the term $\rho^{-1} \cD$ as a transport of $Q$:
\begin{equation}\label{e:QD}
\dD_t Q: = Q_t + u Q_x = \rho^{-1} \cD(u,\rho).
\end{equation}
One classical example is the local dissipation in divergence form
\begin{equation}\label{e:L}
\cD(u,\rho) = (\mu(\rho) u_x)_x.
\end{equation}
This example embodies a broad family of models that appear in various physical phenomena such as  barotropic compressible fluids, slender jets, shallow water waves, etc.  Here, $\mu(\rho)$ designates the dynamic viscosity of the fluid which is typically given by the constitutive power law
\be\label{EOS}
\mu(\rho)=c_\mu \rho^\alpha, \qquad    c_\mu>0, \qquad \alpha \geq 0.
\ee
In this case,
\begin{equation}\label{e:locQ}
Q = \frac{\mu(\rho) \rho_x}{\rho^2}.
\end{equation}
In compressible barotropic fluid models, the pressure equation of state is of the form
\be\label{PEOS}
p(\rho)=c_p \rho^\gamma, \qquad    c_p>0, \qquad \gamma > 0.
\ee
Some special cases include a barotropic monatomic gas with
$\alpha={1/3}$ and $\gamma={5/3}$,  shallow water waves with $\alpha=1$ and $\gamma=2$ and
viscous slender jet dynamics with $\alpha=1$ and $\gamma=1/2$ (but with negative pressure $c_p<0$). See   \cite{CDNP18}, \cite{haspot2014existence} and \cite{mellet2008existence} for recent studies and further discussion.

Recently, another class of examples, referred to as singular Euler-Alignment models, with \emph{non-local} density-dependent dissipation appeared in the context of collective behavior \cite{ST1}:
$$
\cD(u,\rho) = \rho [ \cL^s(\rho u) - u\cL^s(\rho)].
$$
Here $\cL^s := -(-\partial_{xx} )^{s/2}$ for $s\in (0,2)$ is the fractional Laplacian given by the integral representation
\[
\cL^s f(x)  = \int_\T \phi_s (x-y) (f(y) - f(x)) \dy, \qquad \phi_s (x) = c \sum_{k \in \Z} \frac{1}{|x+2\pi k|^{1+s}}.
\]
The transport representation of the dissipation comes as a consequence of the commutator structure in this case, where one finds
\begin{equation}\label{e:flockQ}
	Q = \p_x^{-1} \cL^s \rho.
\end{equation}
With regard to the pressure law, two cases arise naturally from the corresponding agent-based Cucker-Smale  system introduced in \cite{CS2007a,CS2007b}. One is the pressureless case, $p = 0$, which presents itself as a limit to monokinetic concentration $f \to \rho \d_{v = u}$ in kinetic formulation with a strong local alignment forcing, see \cite{FK2019,KV2015} for rigorous study. Second is an isentropic pressure given by
\be\label{PEOSnl}
p(\rho)=c_p \rho, \qquad    c_p \geq 0,
\ee
which arises from stochastically forced systems.  In the strong dissipation limit, the probability density $f$ converges to a Maxwellian distribution, see \cite{KMT1,KMT2,KMT3}.   While the pressureless case is well understood by now and is covered in extensive studies \cite{ST1,ST2,ST3,ST-topo,Do2018}, the pressured case has virtually been omitted in the literature, except for smooth communication \cite{Choi2019}.  Both cases fall under the general class covered in our present study.

Since the relation \eqref{e:QD} is linear, one can access a family of hybrid local/non-local dissipation models as well:
\begin{equation}\label{e:hyb}
\begin{split}
\partial_t (\rho u) + \partial_x (\rho u^2) &=- \partial_x p(\rho) + c_{\nl}  \cD_{\nl}(u,\rho) + c_{\loc}  \cD_{\loc}(u,\rho) +\rho f \\
 \cD_{\loc}(u,\rho) & = (\mu(\rho) u_x)_x, \quad   \cD_{\nl}(u,\rho) = \rho [ \cL^s(\rho u) - u\cL^s(\rho)] 
\end{split}
\end{equation}
In the context of collective behavior these encompass multi-scale alignment models -- classical power law at large scales, and strong singular alignment at local small scales. Although multi-scaling has already appeared on the kinetic level in the analysis of hydrodynamic limit performed in \cite{KMT2}, the net effect of the local alignment considered there averages down to zero in the macroscopic formulation. We argue, however, that local dissipation, along with the plethora of constitutive laws \eqref{EOS}, appears naturally as a singular limit $s \to 2$ of so-called topological model introduced in \cite{ST-topo}. Let us recall the construction.  It is observed in many biological behavioral studies, \cite{Exp2008,Exp2012}, that ``agents", such as birds or fish, probe local environment sensing only a fixed number of other agents around them. Consequently, the actual communication neighborhood is determined by topologies determined by the density of the flock rather than the classical Euclidean one.  The topological density-dependent distance can be defined by the mass of intermediate segment between agents (see \cite{ST-topo} for multi-dimensional construction):
\[
d(x,y) = \left| \int_x^y \rho(z,t) \dz \right|.
\]
Since communication in dense areas progresses slower, the mass-distance should decrease effective viscosity of the alignment, leading one to consider a kernel inversely dependent on $d$:
\[
\phi(x,y) = \frac{h(x-y)}{|x-y|^{1+s - \t} d^\t(x,y)},
\]
where $\t>0$ is a parameter that gauges contribution of the topological part and $h$ is a local cut-off function.  The corresponding operator is given by
\[
\cL^{s,\t} f(x)  = \int_\T \phi(x,y) (f(y) - f(x)) \dy.
\]
As $s \to 2$, the normalized operator formally converges to the  local elliptic operator in divergence form:
\[
(2-s) \cL^{s,\t} f \to (\rho^{-\t} f_x)_x.
\]
Consequently, the alignment term converges to
\begin{equation}\label{key}
(2-s)\rho [ \cL^{s,\t}(u \rho) -   u \cL^{s,\t} \rho ] \to (\rho^{2-\t} u_x)_x.
\end{equation}
So, we obtain an example of the local operator \eqref{e:L} with topological viscosity given by $\mu(\rho) = \rho^{2-\t}$.  To summarize, in the context of flocking,  the hybrid model \eqref{e:hyb} describes a flock driven by a strong local topological alignment and global power law communication.  

Returning to the general discussion, we make a key observation --  once the transport quantity $Q$ is identified, one can rewrite the entire system as a  system of conservation laws with the momentum equation given by the transport equation
\begin{equation}
\dD_t X = - h_x(\rho)  + f, \qquad h'(r):= \frac{p'(r)}{r},
\end{equation}
for the new quantity
\[
X = u+ Q.
\]
We will exploit this structure to prescribe an algorithm of constructing a hierarchy of entropy-like quantities which are extremely useful in studying regularity of such systems. The first member in the hierarchy is given by the well known Bresch-Desjardins entropy \cite{bresch2003existence}
\begin{equation}\label{e:H0}
\cH_0 = \frac12 \int_\T  \rho X^2 \dx + \int_\T \pi_0(\rho) \dx,
\end{equation}
where  $\pi_0$ is the pressure potential given by 
\begin{equation}\label{pi0}
\pi_0(\rho)= \rho \int_{{\bar \rho}}^\rho \frac{p(s)}{s^2} \ds, \quad \text{for some } \bar \rho >0.
\end{equation}
The algorithm, described in detail later in Section \ref{ss:HE}, gives rise to higher order entropies, $\cH_1, \cH_2$, $\dots$, each controlling corresponding higher order derivatives $X_x,X_{xx}$, $\rho_x, \rho_{xx}$, etc. Note that in the pressureless case, in particular, $X_x$ is precisely the ``$e$-quantity" discovered in \cite{CCTT2016}, which determines a threshold for regularity of solutions in the bounded kernel case.  

With the use of this method we present, in a unified way, a range of global existence and continuation results for local, non-local, or hybrid models. Let us state our main results now.

\begin{theorem}[Continuation Criterion]\label{t:mainF} Consider the system \eqref{eq:mass}--\eqref{eq:initv} with 
\be
 \cD(u,\rho) := c_{\nl}  \cD_{\nl}(u,\rho) + c_{\loc}  \cD_{\loc}(u,\rho),
\ee
with equations of state given by \eqref{EOS}, \eqref{PEOS}, and $f \in L^\infty(R^+; C^n)$. Consider the cases
\begin{enumerate}
\item purely non-local: $c_{\loc}=0$,  $c_{\nl}> 0$, in which case we require $s\in (\frac53,2)$, $\g>0$. 
\item purely local:  $c_{\loc}> 0$,  $c_{\nl}= 0$, in which case we require $(\alpha\geq 0,\g>1)$ or $(\a >\frac12,\g>0)$.
\item mixed: $c_{\loc}> 0$,  $c_{\nl}> 0$, in which case we require 
\[
\alpha\geq 0,\g>1, s\in (0,2)\quad \text{ or } \quad \a >\frac12,\g>0, s\in (0,2)\quad \text{ or } \quad\a\geq 0,\g>0, s\in(\frac32,2).
\]
\end{enumerate}
Suppose $(u,\rho) \in H^{m+1-\s} \times H^m$, $m \geq 2$, is a local solution on time interval $(0,T)$. Suppose also that 
\bq\label{e:rmin}
\cl :=\inf_{t\in [0, T^*)}\min_{x\in\T}\rho(x, t)>0.
\eq
Then the solution belongs to the class
\begin{align}
u &\in L^\infty(0,T; H^{m+1-\sigma} ) \cap L^2(0,T; H^m) \label{e:apru3} \\ 
\rho &\in L^\infty(0,T;  H^m)   \cap  L^2(0,T; H^{\sigma/2 + m})  \label{e:aprrho3}
\end{align}
where $\sigma$ is the order of the operator $ \cD(u,\rho)$, and hence can be extended locally beyond $T_*$. 
\end{theorem}
\begin{remark}
We note that the correspondence 
\begin{equation}\label{e:ruorder}
	(\rho \in H^m) \sim (u \in H^{m+1-\sigma})
\end{equation}
 is natural for reasons to be clarified later. 
\end{remark}

The statement about purely local models in our Theorem \ref{t:mainF} provides an alternate proof (and extension) of Theorem 1.1 in \cite{CDNP18}. In that paper, an ``active potential" $w$ which satisfied a less-degenerate parabolic equation was used to propagate higher regularity provided the density nowhere vanished.  In our work, we establish the same result by analyzing the entropy hierarchy.  

Our next result establishes global well-posedness for a class of hybrid models, which is proved by propagating a lower bound on the density and appealing to Theorem \ref{t:mainF}.

\begin{theorem}[Global Existence]\label{t:mainF2} Assume
 $f \in L^\infty(\R^+; C^n)$, $p\in C^{n+1}_\loc(\R^+)$ with $p'(r) >0$ for any $r>0$ and
\be
c_{\nl} \geq 0, c_\loc >0, \quad \alpha\in(0,1/2).
\ee
Then any given local classical solution with non-vacuous initial data enjoys a priori lower bound \eqref{e:rmin} on its interval of existence. Consequently, any non-vacuous initial condition $(u_0,\rho_0) \in H^{m+1-\s} \times H^m$, $m \geq 2$, gives rise to a unique global solution in the range of parameters stated in Theorem \ref{t:mainF}. 
\end{theorem}

In the purely local case, Theorem \ref{t:mainF2} provides an alternate proof to that of Mellet and Vasseur  \cite{mellet2008existence} who proved global well-posedness in the parameter range $\alpha <1/2$ and $\gamma>1$.
As another application, we obtain global existence for collective behavior models.

\begin{corollary}
	Any collective behavior model, $\g=1$, with multi-scale diffusion $c_{\nl}, c_\loc > 0$  in the range of parameters $\a \in(0,1/2)$, $s\in (3/2,2)$ is globally well-posed for initial data in the class  $(u_0,\rho_0) \in H^{m-1} \times H^m$, $m \geq 2$.
\end{corollary} 

Our final result concerns the long-time behavior of the velocity and density fields in models which possess a non-local dissipation component.  In particular, we show that the energy inequality together with the non-local analogue of Bresch-Desjardins entropy  imply flocking in an $L^2$-sense.

\begin{theorem}[Non-local ``Second Law" Implies Flocking]\label{t:mainF3} Consider the forceless system \eqref{eq:mass}--\eqref{eq:initv} with $c_{\nl}>0$, $c_\loc \geq 0$ and pressure law given by \eqref{PEOSnl}. Then any classical solution undergoes flocking behavior in the weighted $L^2$-sense:
\begin{equation}
\int_{\T\times \T} |u(x)-u(y)|^2 \rho(x) \rho(y)\dx \dy  + \left\|\rho(t) -\fint\rho_0 \dx\right\|_{L^1(\mathbb{T})}^2 \lesssim \frac{\ln t}{t}.
\end{equation}
\end{theorem}
We note that that in the pressurized case, as opposed to pressureless, the density always converges to a uniform state selected by its average. It is an indirect consequence of stochastic diffusion that leads to persistent mixing and eventual homogenization of the flock density. Analogous behavior was also observed in the study of Choi \cite{Choi2019}, under the assumption of globally bounded velocity field $u$. Note, however, that such assumption is not guaranteed a priori due to lack of the maximum principle in the pressured system.

\section{Hierarchy of Entropies Method}\label{ss:HE}

As observed in the Introduction, due to the transport formulation of the dissipation \eqref{e:QD} one can rewrite the entire momentum equation \eqref{eq:mom} as a transport equation for the new quantity
\[
X = u + Q,
\]
\begin{equation}
\dD_t X = - h_x(\rho)  + f, \qquad h'(r):= \frac{p'(r)}{r}.
\end{equation}
Now \eqref{eq:mass}-\eqref{eq:mom} becomes as a system of conservation law for the new pair of unknowns $(\rho, \rho X)$.  We will exploit this structure to prescribe an algorithm of constructing a hierarchy of entropy-like quantities. First, let us make a general observation  --  if we have two quantities, $X$ and $\rho$, one is transported and the other is conserved 
\[
\dD_t X = 0, \quad \rho_t+ (u \rho)_x = 0,
\]
then the ``energy" given by $\frac12 \int_\T \rho X^2 \dx$ is preserved for all time. In the presence of the pressure such conservation is destroyed,
\[
\ddt \frac12 \int_\T \rho X^2 \dx = -\int_\T  X p_x(\rho)\dx + \int_\T \rho f X \dx,
\]
and the pressure term splits into two elements coming from $X$
\[
X p_x(\rho) = u p_x(\rho) + Q p_x(\rho).
\]
It turns out that the $Q$-term is in fact dissipative in all the examples we considered so far. So the only term that needs to be eliminated is $u p(\rho)_x$. This is done with the use of the pressure potential given by
\begin{equation}
\pi_0(\rho)= \rho \int_{{\bar \rho}}^\rho \frac{p(s)}{s^2} \ds, \quad \text{for some } \bar \rho >0.
\end{equation}
Indeed, 
\begin{equation}\label{e:pi0}
\ddt  \int_\T \pi_0(\rho) \dx =  \int_\T u p_x(\rho) \dx.
\end{equation}
We thus recover what is known as  Bresch-Desjardins's entropy
\begin{equation}\label{e:H0}
\cH_0 = \frac12 \int_\T  \rho X^2 \dx + \int_\T \pi_0(\rho) \dx.
\end{equation}
According to the computations above we obtain the following balance relation
\begin{equation}\label{e:H0bal}
\ddt \cH_0 =  \int_\T Q_x p(\rho) \dx + \int_\T \rho f X \dx.
\end{equation}
In parallel, due to \eqref{e:pi0}, we obtain an energy balance relation for the energy of the system given by 
\be \label{e:e}
\cE =  \frac12 \int_\T  \rho |u|^2 \dx + \int_\T \pi_0(\rho) \dx,
\ee
\begin{equation}\label{e:energy}
\ddt \cE = \int_\T u \cD(u,\rho) \dx + \int \rho u f \dx.
\end{equation}
In all cases of interest the pressure term in \eqref{e:H0bal} is sign-definite provided $p'(r) \geq 0$.  Indeed, in the non-local case \eqref{e:flockQ} we obtain
\begin{equation}\label{e:dissFlock}
\int_\T Q_x p(\rho) \dx = \int_\T p(\rho) \cL^s \rho \dx = - \frac12 \int_{\T^2} \phi_{s,\rho}(x,y) (\rho(y) - \rho(x))^2 \dx \dy \leq 0,
\end{equation}
where 
\[
\phi_{s,\rho}(x,y) = \phi_s (x,y)\int_0^1 p'(\th \rho(x) + (1-\th) \rho(y)) \dth.
\]
In the local case \eqref{e:locQ} we find
\begin{equation}\label{e:dissloc}
\int_\T Q_x p(\rho) \dx =  - \int_\T \frac{\mu(\rho) \rho^2_x p'(\rho)}{\rho^2} \dx \leq 0.
\end{equation}
Consequently, the initial entropy $\cH_0$ gives control over $\rho X^2 \in L^\infty_t L^1_x$.  Together with the energy conservation, this in turn controls the solo-density term:
\[
\int_\T  \rho |Q|^2 \dx \leq \int_\T  \rho |u|^2 \dx + \int_\T  \rho X^2\dx
\]
which will be used to extract initial regularity information on the density in each of local and non-local cases separately. 

As to the hybrid case, we have $Q = Q_{\nl} + Q_\loc$. The key observation is that we can extract control on each of the terms separately. Indeed, writing
\[
\int_\T \rho|Q_{\nl}+Q_\loc|^2  \dx =\int_\T  \rho |Q_{\nl}|^2 \dx+ \int_\T  \rho |Q_{\loc}|^2 \dx  +  2 \int_\T  \rho Q_{\nl} Q_{\loc}\dx,
\]
we observe that the integral of the cross-dissipation is non-negative:
\begin{align}
\int_\T  \rho Q_{\nl} Q_{\loc}\dx =  \int_\T    \p_x^{-1} \cL^s \rho  \frac{\mu(\rho) \rho_x}{\rho}\dx=  \int_\T    \p_x^{-1} \cL^s \rho\psi(\rho)_x\dx =  -\int_\T    \psi(\rho)\cL^s \rho\dx \geq 0\label{poscross}
\end{align}
where $\psi'(r) ={\mu(r)}/{r}$. The non-negativity holds, in fact, provided $\mu(r) \geq 0$, which is manifestly true for positive viscosities.

Coming back to the entropy construction we now present the next step in the hierarchy.   Noting that in the pressureless case the quantity $X_x$ would have satisfied the continuity equation, and hence, in combination with the density the new variable $Y= \frac{1}{\rho} X_x$ would have been transported.  By analogy with the previous, we would then start construction with the pressureless term $\rho Y^2 = \rho^{-1} X_x^2$. Note that $X_x$ satisfies
\begin{equation}\label{e:eeq}
\p_t X_x + (u X_x)_x = -h_{xx}(\rho) +f_x.
\end{equation}
Hence, in conjunction with mass conservation,
\[
\dD_t Y = - \rho^{-1} h_{xx}(\rho) + \rho^{-1} f_x.
\]
We thus  obtain
\begin{equation}\label{e:balX}
\ddt \frac12\int_\T \rho Y^2\dx = - \int_\T  h_{xx}(\rho) Y \dx + \int_\T  f_x Y \dx.
\end{equation}
The appropriate next order pressure potential that eliminates the first term on the right hand side is given by
\begin{equation}
\pi_1(\rho,\rho_x) = \frac12 h'(\rho) \frac{ \rho_x^2}{\rho^2}.
\end{equation}
The details of this computation will be provided in the sections below.  We thus arrive at the next entropy
\begin{equation}\label{e:H1}
\cH_1 =  \frac12\int_\T \rho Y^2 \dx + \int_\T 	\pi_1(\rho,\rho_x) \dx.
\end{equation}
the algorithm is now clear. For the second order entropy we denote $Z = \frac{1}{\rho} Y_x$ and define
\begin{equation}\label{e:H2intro}
\begin{split}
\cH_2 &=  \frac12\int_\T \rho Z^2 \dx + \int_\T 	\pi_2 \dx \\
\pi_2 & = \frac12 h'(\rho) \frac{\rho_{xx}^2}{\rho^4}.
\end{split}
\end{equation}
Continuing in the same fashion we can design an entropy-like quantity of any order, where the pressure potential is given by 
\[
\pi_n  = \frac12 h'(\rho) \frac{(\p^n_x\rho)^2}{\rho^{2n}},
\]
while the kinetic term is constructed inductively,
\[
X_n = \frac{1}{\rho}\p_x X_{n-1}.
\]
We form the $n$-th entropy accordingly,
\[
\cH_n =  \frac12\int_\T \rho X_n^2 \dx + \int_\T 	\pi_n \dx.
\]

With the help of this hierarchy we establish a direct control over any higher order regularity of solution, consistent with that of the initial datum, where the relative smoothness of $u$ and $\rho$ mentioned in \eqref{e:ruorder} naturally equilibrates the order of and the velocity and density terms as they enter into an expression for $X_n$.   It should be noted however that these entropies do not decay precisely as the first element $\cH_0$. Instead, they satisfy  ODEs with residual terms. 
Controlling those residual terms presents the main technical component of the method.

\section{Global Existence and Flocking}

We start by presenting less technical proofs of Theorem \ref{t:mainF2} and \ref{t:mainF3}. We begin by remarking that the proof of Theorem \ref{t:mainF2} along with the continuation criterion Thm \ref{t:mainF} require local well-posedness of the model equations:
\begin{proposition}[Local Well-Posedness]\label{prop:local}
	Let $c_\loc\geq 0$ and $c_\nl \geq 0$. Assume that $p:\mathbb{R}^+\to \mathbb{R}$ and $\mu:\mathbb{R}^+\to \mathbb{R}^+$ are $C^\infty$ functions away from zero. Assume $s\in (0,2)$ and let $\sigma := 2$ if $c_\loc >0$ and $\sigma := s$ if $c_\nl>0$ and $c_\loc=0$.
	Let $(u_0,\rho_0) \in H^{m+1-\sigma} \times H^m$, $m \geq 2$,  such that $r_0 := \min_{x\in \T}\rho_0 > 0$. 
	Suppose that for all $T>0$
	\begin{equation*}
	f\in L^2(0,T; H^{m-\sigma}(\T)).
	\end{equation*}
	Then, there exists a time $T_0>0$ depending only on $\| (\rho_0, u_0)\|_{H^m(\T)\times H^{m+1-\sigma}(\T)}$, { $r_0$ and $f$}, and a unique strong solution $(\rho, u)$ to \eqref{eq:mass}-\eqref{eq:initv} on $[0, T_0]$ with data $(\rho_0,u_0)$ such that 
	\begin{align}
u &\in C(0,T_0; H^{m+1-\sigma} (\T)) \cap L^2(0,T_0; H^m(\T)) \\ 
\rho &\in C(0,T_0;  H^m(\T))   \cap  L^2(0,T_0; H^{\sigma/2 + m}(\T))  
\end{align}
	and $\rho(x, t)>\frac{r_0}{2}$ for all $(x, t)\in \T\times [0,T_0]$. \end{proposition}
	It should be remarked that this local well-posedness result covers all cases discussed in Theorems \ref{t:mainF} and \ref{t:mainF2} but it holds in far greater generality.  In particular, we do not require power-law forms for the pressure and viscosity constitutive laws.  
We do not produce a proof of Proposition \ref{prop:local} here, which is standard.  In fact, the purely local case was established in Appendix II of  \cite{CDNP18}.  On the other hand, the purely nonlocal case follows from the a general proof which works in higher dimensions and is provided in Appendix A of \cite{KMT3}, see also \cite{ST1} for singular kernel pressureless case. The mixed case is a  routine exercise, so is omitted.  With local well-posedness in hand, we proceed with the proof of global well-posedness.

\begin{proof}[Proof of Theorem \ref{t:mainF2}]
By  Prop. \ref{prop:local}, we have a local strong solution on some interval $[0,T_0]$. We aim to show that $T_0$ may be taken infinite by establishing a lower bound on the density and appealing to the no-vacuum continuation criteria established in Theorem \ref{t:mainF}.

Let us recall from the previous section that either in the local-only or in hybrid cases we establish control over the local term $\int_\T \rho |Q_\loc|^2 \dx$ the entropy $\cH_0$ and energy $\cE$. Both are bounded uniformly in time due to \eqref{e:H0bal} and \eqref{e:energy}, where we can estimate the force term by 
\begin{equation}\label{e:auxfur}
\begin{split}
\left| \int_\T f\rho u \dx\right| & \leq |f|_\infty \left( \cM + \int_\T \rho |u|^2 \dx \right) \leq C_1 + C_2 \cE, \\
\left| \int_\T f\rho X \dx \right| & \leq |f|_\infty \left( \cM + \int_\T \rho |X|^2 \dx \right) \leq C_1 + C_2 \cH_0.
\end{split}
\end{equation}
For $\mu(r)= r^{\a}$ with $\a <1/2$ this implies 
\[
\int_\T \rho |Q_\loc|^2 \dx = \int_\T \rho \left| \rho^{\a-2} \rho_x \right|^2 \dx = \int_\T |(\rho^{\a - \frac12})_x |^2 \dx = \| \rho^{\a - \frac12} \|_{\dot{H}^1} < \infty.
\]
To establish pointwise bound we simply recall that the density has a conserved finite mass $\|\rho(t)\|_1=\cM>0$. So, for each time $t$ there exist a point $x_0(t)$ such that $\rho(x_0(t),t)> \cM/2$. We find 
\be
\int_{x_0(t)}^x (\rho^{\a-1/2})_x  \dx= \rho^{\a-1/2}(x,t)- \rho^{\a-1/2}(x_0(t)) \geq  \rho^{\a-1/2}(x,t)- (\cM/2)^{\a-1/2}.
\ee
It follows that $ \rho^{\a-1/2}\in L^\infty_tL^\infty_x$ which implies  $1/\rho\in L^\infty_tL^\infty_x$ since $\a<1/2$.  This finishes the proof.
\end{proof}

\begin{proof}[Proof of Theorem \ref{t:mainF3}]

Due to the energy-entropy law elucidated in the previous section, \eqref{e:H0bal}, \eqref{e:energy}, and the specific form of enstrophy coming from the non-local dissipation, we obtain
\begin{equation}
\begin{split}
\ddt \cH_0 & \leq - c_1 \int_{\T\times \T} \phi_{s}(x-y) (\rho(y) - \rho(x))^2 \dx \dy  \\
\ddt \cE & \leq  - c_2 \int_{\T\times \T} \phi_s (x-y)|u(x)-u(y)|^2 \rho(x) \rho(y)\dy \dx.
\end{split}
\end{equation}
Note that under the linear pressure law \eqref{PEOSnl}, $\phi_s = \phi_{s,\rho}$. Moreover, by the  Gallilean invariance of the system and conservation of momentum, we may assume  that the total momentum remains $0$:
\[
\int_\T \rho u \dx = 0.
\]
Let us also assume $\cM = 1$. From the entropy dissipation we have 
\[
\int_{\T\times \T} \phi_{s}(x-y) (\rho(y) - \rho(x))^2 \dx \dy 
\geq c_0 |\rho - 1|_2^2 \geq c_0 \int_{\T} \rho \log \rho \dx \geq c_0 \int_{\T} \pi_0(\rho) \dx.
\]
This shows that $\int_0^\infty \int_{\T} \pi_0 \dx \dt <\infty$. Then 
\begin{equation}\label{e:envar}
\frac12 \int_{\T\times \T} |u(x)-u(y)|^2 \rho(x) \rho(y)\dy \dx = \cM \int_\T \rho |u|^2 \dx = c_1\cE - c_2 \int_\T \pi_0 \dx.
\end{equation}
Hence,
\[
\ddt \cE \leq - c_1 \cE + F(t),
\]
where $F \in L^1(\R_+)$.  By Duhamel, we obtain 
\[
\cE(t) \leq e^{-c_1 t} \cE_0 + \int_0^t e^{-c_1(t-s)} F(s) \ds.
\]
The asymptotics of convergence  $\cE \to 0$ is based on the convolution integral. We can estimate it as follows. Let us define a sequence of times by
\[
t_{m+1} = t_m + \frac{\l}{c_1} \ln m, \text{ for some } \l>1.
\]
Then $t_m \sim  \frac{\l}{c_1} \ln( m!)$, and by Stirling approximation, $t_m \sim m \ln m$.   Let $K = \int_0^\infty F(s) \ds$.  Then for every natural $n\in \N$ there exists an $m \in [n, (K+1)n]$ such that 
\[
\int_{t_{m-1}}^{t_m} F(s) \ds \leq \frac{1}{n}.
\]
Indeed, otherwise, $\int F\ds >K$.  At time $t_m$ we then have an estimate
\[
\cE(t_m) \leq \frac{1}{m^{cm}}  +  \int_0^{t_{m-1}} e^{-c_1(t_m-s)} F(s) \ds + \int_{t_{m-1}}^{t_m} e^{-c_1(t_m-s)} F(s) \ds \leq \frac{1}{m^{cm}}  +  \frac{K}{m^\l}  + \frac{1}{n}.
\]
But $1/n \sim 1/m$ which appears to be the leading order term.  Recalling that $t_m \sim m \ln m$ we conclude that $1/m \lesssim \ln t_m / t_m$. So, we have 
\[
\cE(t_m) \lesssim \frac{\ln t_m }{t_m}.
\]
For any other $t_m<t <t_{m+1}$, we have by monotonicity of the energy 
\[
\cE(t) \leq \cE(t_m) \lesssim  \frac{\ln t_m }{t_m} \sim  \frac{\ln t_{m+1} }{t_{m+1}} \leq  \frac{\ln t}{t} .
\]
This establishes the desired asymptotic. 

Finally, by the Csiszar-Kullback inequality, $\int \pi_0 \dx \geq |\rho-1|_1^2$, and \eqref{e:envar}, we obtain a flocking statement in the weighted $L^2$-sense:
\begin{equation}
\int_{\T\times \T} |u(x)-u(y)|^2 \rho(x) \rho(y)\dy \dx  + |\rho -1|_1^2 \lesssim \frac{\ln t}{t}.
\end{equation}
\end{proof}

\section{Continuation of non-vacuous solutions}
This section contains main technical ingredients of the hierarchy method, and provides the proof of Theorem \ref{t:mainF}.

Consider the evolution equations
\begin{equation}\label{e:F}
\begin{split}
\partial_t \rho + \partial_x (u \rho ) &= 0,  \\
\partial_t (\rho u) + \partial_x (\rho u^2) &=- \partial_x p(\rho) +c_{\nl}  \cD_{\nl}(u,\rho) + c_{\loc}  \cD_{\loc}(u,\rho)+\rho f 
\end{split}
\end{equation}
where, recall, the nonlocal and local dissipative operators are defined as
\[
 \cD_{\nl}(u,\rho):=\rho [ \cL^s(\rho u) - u\cL^s(\rho)], \qquad  \cD_{\loc}(u,\rho)=  (\mu(\rho) u_x)_x.\nonumber
\]
We will be considering $c_{\nl}\geq 0$ and $c_{\loc}\geq 0$.  The discussion of various cases requires us to break up our argument.  
In all cases, we consider the power-law for the pressure, i.e. $p(\rho)= c_p \rho^\gamma$ for some $\gamma > 0$ and  $ c_p>0$. In this case we have
 the explicit formula
 \be\label{formula:pi}
\pi_0(\rho)= c_p \rho \int_{\bar{\rho}}^\rho s^{\gamma-2}\rmd s =\begin{cases}
\frac{c_p }{\gamma-1}\rho^{\gamma}  &\quad \gamma >1,\ \bar{\rho}=0,\\
c_p \rho \log(\rho/\bar{\rho})  &\quad \gamma = 1, \ \bar \rho =\frac{1}{2\pi} \cM \\
\frac{c_p }{1-\g}(1 - \rho^{\gamma} )&\quad \gamma  < 1, \ \bar \rho =1.
\end{cases}
\ee 
Thus, if $\gamma>1$ then $\pi_0(\rho)\geq 0$ is non-negative pointwise. If $\gamma=1$, which is of particular relevance in the context of flocking \eqref{PEOSnl}, then upon spatial integration it is non-negative by the Csiszar-Kullback inequality, i.e. $\int \pi_0 \dx \geq |\rho-\bar{\rho}|_1^2$. This non-negativity will be repeatedly used for extracting information from the a priori estimates arising from the Bresch-Desjardins entropy balance. When $\g<1$, we simply note that $\int \pi_0 \ds$ is bounded by the mass.

Assume that  $(u,\rho)$ is a smooth solution on the time interval $[0, T^*)$ such that  for any $m \geq 2$
\begin{align}
u &\in L^\infty(0,T; H^{m+1-\sigma} ) \cap L^2(0,T; H^m) \\ 
\rho &\in L^\infty(0,T;  H^m)   \cap  L^2(0,T; H^{\sigma/2 + m}) 
\end{align}
for any $T<T^*$,
where we have introduced
\be
\sigma = \begin{cases}2 & c_{\nl}\geq 0, c_\loc>0\\ s & c_{\nl}>0, c_\loc=0 \end{cases}.
\ee
Assume also that no vacuum has appeared
\be\label{e:rmin}
\cl :=\inf_{t\in [0, T^*)}\min_{x\in\T}\rho(x, t)>0.
\ee

\subsection{\bf Mass, Energy and $\cH_0$-entropy}

First, from the continuity equation \eqref{eq:mass} total mass is conserved
\be\label{massConsbnd}
\cM = \Vert \rho(\cdot, t)\Vert_{L^1(\T)}=\Vert \rho_0\Vert_{L^1(\T)}.
\ee
Recall next the basic energy balance \eqref{e:energy}, which in this case reads
\begin{equation} \label{energyBalance}
\ddt \cE = - c_{\nl}\int_{\T\times \T}  \frac{1}{2}\phi_s (x-y)|u(x)-u(y)|^2 \rho(x) \rho(y)\rmd y \rmd x- c_\loc \int_{\T} \mu(\rho) |u_x|^2\rmd x+\int_\T f\rho u\rmd x.
\end{equation}
holds for any $t\in [0, T^*)$. In the view of the estimate \eqref{e:auxfur} this establishes uniform control in the energy space $u\in L^\infty L^2 \cap L^2 H^{\sigma/2}$ on the given time interval.  

The Bresch-Desjardins entropy \eqref{e:H0}  satisfies the balance
 \eqref{e:H0bal}, \eqref{e:dissFlock},
\begin{equation}
\ddt \cH_0 =  -c_{\nl} \int_{\T^2} \frac{1}{2} \phi_{s,\rho}(x,y) (\rho(y) - \rho(x))^2 \dx \dy  -c_\loc \int_{\T^2} \frac{\mu(\rho)p'(\rho)}{\rho} |\rho_x|^2 \dx  + \int_\T \rho f X \dx.
\end{equation}
holds for any $t\in [0, T^*)$.   Again, using  \eqref{e:auxfur}, we find that $\cH_0$ remains bounded.  We now derive conclusions from these balance in two cases:
\\

\noindent \textbf{Purely non-local dissipation:}
Given the finite energy and controllability of $\int \pi_0 \dx$ discussed above, we obtain an $L^2$ bound on $Q = \p_x^{-1} \cL^s \rho$, hence, $\rho \in L^\infty H^{s-1}$. Further condition coming from the dissipation term $\rho \in L^2 H^{s/2}$ follows from the lower bound on the density and $p'$ under the assumptions of Theorem~\ref{t:mainF}.
Provided $s>3/2$, we have the embedding $L^1 \cap \dot{H}^{s-1} = H^{s-1} \ss  L^\infty$, and so, the density is uniformly bounded: $\rho \in L^\infty(0,T; L^\infty)$.
\\

\noindent \textbf{Purely local or  mixed  dissipation:}
When the local or both components are active, the uniform bound on the density comes from several sources. Indeed, in light the non-negativity of the mixed term \eqref{poscross} discussed in the introduction, we have control over the two terms $\int \rho Q_{\nl}^2\rmd x$ and $\int\rho Q_\loc^2\rmd x$  separately from the Bresch-Desjardins entropy.  Boundedness of the density is then established in the following parameter regimes.

\begin{enumerate}
\item  $c_{\nl}>0$ and $c_{\loc}\geq 0$ with $\gamma > 0$, $s>3/2$ and $\alpha\geq 0$.  Then $\rho \in L^\infty(0,T; L^\infty)$ by the non-local argument above. 
\item 
$c_{\nl}\geq 0$  and $c_{\loc}> 0$ with $\gamma > 0$,  $s\in (0,2)$ and $\alpha>1/2$. Following the proof of 
Theorem \ref{t:mainF2}, we find 
\[
\int_\T \rho |Q_\loc|^2 \dx = \| \rho^{\a - \frac12} \|_{\dot{H}^1} < \infty.
\]
Boundedness of the density follows, as in the Theorem, from the conserved finite mass.
\item  $c_{\nl}\geq 0$  and $c_{\loc}> 0$ with  $\gamma>1$, $s\in (0,2)$ and $\alpha >0$. 
 This follows from the bound
\be
\int \left|\partial_x\left( \rho^{\frac{\gamma-1}{2}}\right)\right| \rmd x\leq \|\rho\|_{L^\gamma}^{\gamma/2}\left(\int \rho^{-3} |\rho_x|^2 \rmd x\right)^{1/2} \leq c\|\rho\|_{L^\gamma}^{\gamma/2}\left(\int \rho Q_\loc^2\rmd x\right)^{1/2}<\infty
\ee
with a constant $c:= c(\cl)$. Thus, since $\rho\in L^\gamma$ it follows that $\rho^{\frac{\gamma-1}{2}}\in L^1$ so that combined with the above we have $\rho^{\frac{\gamma-1}{2}}\in L^\infty(0,T;W^{1,1}(\mathbb{T}))$.  Boundedness follows from Sobolev embedding. 
\end{enumerate}
Once an upper bound on the density is established in any of the above cases, then the local part of the BD entropy gives the control $\rho \in L^\infty(0,T; H^1)$.  This control trumps what can be obtained from the dissipation of the BD entropy.

All the a priori bounds we have obtained so for in either of the cases can be summarized as follows:
\begin{align}
u &\in L^\infty(0,T; L^2) \cap L^2(0,T; H^{\sigma /2}) \label{e:apru1} \\ 
\rho &\in L^\infty(0,T; L^\infty)  \cap L^\infty(0,T;  H^{\sigma-1})   \cap  L^2(0,T; H^{\sigma /2})  \label{e:aprrho1}
\end{align}
with norms depending on $\cl$.
\\

\subsubsection{\bf $\cH_1$-entropy and its consequences} We now work out a detailed balance relation for the $\cH_1$-entropy. Let us recall the definitions
\[
\begin{split}
\cH_1 &=  \frac12\int_\T \rho Y^2 \dx + \int_\T 	\pi_1(\rho,\rho_x) \dx\\
Y& = \frac{1}{\rho}( u_x + Q_x), \quad  \pi_1(\rho,\rho_x) = \frac12 h'(\rho) \frac{ \rho_x^2}{\rho^2}.
\end{split}
\]
According to \eqref{e:balX}, \eqref{e:H1} we have
\begin{equation}\label{e:H1bal}
\ddt \cH_1 =  - \int_\T  h_{xx}(\rho) Y \dx + \ddt   \int_\T 	\pi_1(\rho,\rho_x) \dx + \int_\T  f_x Y \dx.
\end{equation}
Looking ahead at the argument below, we remark that expansion of the pressure potential term $\pi_1$ on the right hand side of this enstrophy budget  will produce a dissipation term:
\[
- c_{\nl} \| \rho\|_{H^{\frac{s}{2} + 1}}^2 - c_\loc \| \rho\|_{H^{2}}^2.
\]
 We will be using it repeatedly to absorb various  residual terms by interpolation using uniform bounds on the density from below, above, and in $H^{\sigma-1}$ by \eqref{e:aprrho1}. 

So, first, let us estimate the forcing 
\[
\left|   \int_\T  f_x Y \dx  \right| \leq c(|f|_{C^1}) + \cH_1.
\]
Next, we expand the $Y$-term:
\begin{equation}\label{e:hxxY}
- \int_\T  h_{xx}(\rho) Y \dx = - \int_\T  h'(\rho) \rho_{xx} Y \dx- \int_\T  h'' (\rho)\rho_x^2  Y \dx.
\end{equation}
The second term can be estimated by 
\begin{equation}\label{e:Yrho2a}
\left|   \int_\T  h'' (\rho)\rho_x^2  Y \dx  \right| \leq    \cH_1^\frac12 |\rho_x|_4^2   ,
\end{equation}
and using that $|\rho_x|_4 \leq \|\rho\|_{\dot{H}^{\sigma/2+1}}^{\frac{9-4\sigma}{2(4-\sigma)}} $, and $\frac{9-4\sigma}{4-\sigma} \leq 1$ as long as $\sigma \geq \frac53$, we obtain
\begin{equation}\label{e:Yrho2b}
\left|   \int_\T  h'' (\rho)\rho_x^2  Y \dx  \right| \leq   c(\e) \cH_1 + \e \| \rho\|_{H^{\frac{\sigma}{2} + 1}}^2.
\end{equation}
In view of the remark above we can absorb the term $ \e \| \rho\|_{H^{\frac{\sigma}{2} + 1}}^2$ into the upcoming dissipative contribution from the pressure potential.  The remaining residual term $- \int_\T  h'(\rho) \rho_{xx} Y \dx$ will in fact be completely canceled out by another contribution of the pressure potential on which we focus for the remainder of the proof. First, we introduce  a couple of shortcuts that greatly simplify the exposition. 

 \begin{itemize}
 \item  Throughout this proof, we will routinely drop integral signs for brevity.  All equalities are intended to hold only upon spatially integrating over $\mathbb{T}$.
 \item We denote $g(r) = \frac12 h'(r)/r^2$ so that $\pi_1 = g(\rho) \rho_x^2$. 
 \end{itemize}
Let us compute the potential now
\begin{align} \nonumber
\ddt \pi_1 & = - g'  \rho_x^2 (u\rho)_x - 2 g \rho_x (u\rho)_{xx} = g' \rho_x^2 (u\rho)_x  + 2 g \rho_{xx} (u\rho)_x \\
& = 2g \rho \rho_{xx} u_x+ \rho g' \rho_x^2 u_x + g' \rho_x^3 u + 2 g \rho_{xx} \rho_x u \nonumber
\end{align}
{integrating by parts in the last term}
\begin{align} \nonumber
& = 2g \rho \rho_{xx} u_x+ \rho g' \rho_x^2 u_x + g' \rho_x^3 u - g' \rho_x^3 u - g \rho_x^2 u_x\\
& = 2g \rho \rho_{xx} u_x + \rho g' \rho_x^2 u_x -  g \rho_x^2 u_x. \label{e:lastpi1}
\end{align}
The last two terms are of the  form $q(\rho) \rho_x^2 u_x$, where $q$ is a smooth function on $\R^+$. We can estimate any such term by replacing 
\be
u_x = \rho Y - c_{\nl}\cL^s \rho - c_\loc( \mu(\rho)\rho_x/\rho^2 )_x.
\ee
 The residual term $q(\rho)\rho \rho_x^2 Y$ enjoys the same estimate  as in \eqref{e:Yrho2b}.   The local term breaks up into two: $q(\rho) \rho_x^2\rho_{xx}$ and $q(\rho) \rho_x^4$ for smooth $q$ (which we redefine line by line). Interpolating between $H^2$ and $H^1$ we obtain
 \be\label{locterm1}
 \left| \int_\T q(\rho)  \rho_x^2\rho_{xx} \dx \right| \leq  |\rho_x|_\infty \| \rho\|_{{H}^{1}} \| \rho\|_{{H}^{2}}\leq    \e \| \rho\|_{\dot{H}^{2}}^2 + c(\e),
 \ee
 and using  $|\rho_x|_4 \leq \|\rho\|_{\dot{H}^{2}}^{\frac{1}{4}} $ we have
 \be\label{locterm2}
 \left| \int_\T q(\rho)  \rho_x^4 \dx \right| \leq  |\rho_x|_4^4\leq \|\rho\|_{\dot{H}^{2}}  \leq    \e \| \rho\|_{\dot{H}^{2}}^2 + c(\e).
 \ee
To estimate the non-local part $q(\rho) \rho_x^2 \cL^s \rho$, we symmetrize in the integral representation of $\cL^s$ and estimate according to the following
\begin{equation}\label{e:aux10}
\begin{split}
\left| \int_\T q(\rho) \rho_x^2 \cL^s \rho \dx \right|  \leq  |\rho_x|_\infty^2 \| \rho \|_{H^{s/2}}^2 + |\rho_x|_\infty \| \rho \|_{H^{s/2+1}} \| \rho \|_{H^{s/2}}.
\end{split}
\end{equation}
By the Gagliardo-Nirenberg inequality, 
\[
|\rho_x|_\infty \leq  \| \rho\|_{\dot{H}^{s/2+1}}^{\frac{5-2s}{4-s}}  \| \rho\|_{\dot{H}^{s-1}}^{\frac{s-1}{4-s}}, \qquad  \| \rho\|_{\dot{H}^{s/2}} \leq \| \rho\|_{\dot{H}^{s/2+1}}^{\frac{2-s}{4-s}} \| \rho\|_{\dot{H}^{s-1}}^{\frac{2}{4-s}}.
\]
Thus,
\[
\left| \int_\T q(\rho) \rho_x^2 \cL^s \rho \dx \right| \leq  \| \rho\|_{\dot{H}^{s/2+1}}^{\frac{14-6s}{4-s}}  + \| \rho\|_{\dot{H}^{s/2+1}}^{\frac{11-4s}{4-s}} \leq \e \| \rho\|_{\dot{H}^{s/2+1}}^2 + c(\e),
\]
due to both powers being less than or equal $2$ as long as $s >\frac32$. 

Finally, for the first term on the right hand side of \eqref{e:lastpi1} we  have
\[
2g \rho \rho_{xx} u_x = 2g \rho^2 \rho_{xx} Y - 2c_{\nl}g \rho \rho_{xx} \cL^s \rho - 2c_\loc g \rho \rho_{xx}( \mu(\rho)\rho_x/\rho^2 )_x.
\]
Notice that $2g \rho^2 \rho_{xx} Y = h'(\rho) \rho_{xx} Y$, which cancels with the first term on the right hand side of \eqref{e:hxxY}. The  main contribution of the last two terms is dissipation.  Indeed, omitting constants and integrating by parts,
\[
- g \rho \rho_{xx} \cL^s \rho =  g' \rho_x^2 \cL^s \rho + g \rho_x \cL^s \rho_x.
\]
The first one we already estimated. The second, after symmetrization, is bounded by 
\[
 g \rho_x \cL^s \rho_x \leq - c_0 \| \rho\|_{\dot{H}^{s/2+1}}^2 + |\rho_x|_\infty \| \rho \|_{H^{s/2+1}} \| \rho \|_{H^{s/2}},
\]
with the latter already being treated in \eqref{e:aux10}. The last local term splits into
\be
 - g \rho \rho_{xx}(\mu(\rho)\rho_x/\rho^2 )_x =  -q_1 |\rho_{xx}|^2 + q_2 |\rho_x|^2 \rho_{xx}, \quad q_1>0.
\ee
With the use of \eqref{locterm1}, we estimate it by 
\[
 - g \rho \rho_{xx}(\mu(\rho)\rho_x/\rho^2 )_x \leq - c  \| \rho \|_{H^{2}}^2 + c(\e). 
\] 
Collecting the estimates we obtain
\begin{equation}\label{e:H1bal2}
\ddt \cH_1 \leq  c_1 \cH_1 + c_2  - c_3 \| \rho\|_{\dot{H}^{\sigma/2+1}}^2.
\end{equation}
This implies  a uniform bound on $\cH_1$ on the time interval at question along with integrability of $  \| \rho\|_{\dot{H}^{\sigma/2+1}}^2$. As a consequence of the  positivity of $\pi_1$, we obtain $\rho \in L^\infty H^1$, and 
\be
\rho Y= u_x + c_{\nl} \cL^s\rho +  c_\loc( \mu(\rho)\rho_x/\rho^2 )_x\in L^\infty L^2.
\ee 
In the purely non-local case ($c_\loc=0$), if we apply $1-s$ derivatives on this expression we still obtain a function in $L^2$, yet $\rho_x \in L^2$ by the previous. This places  $u$ into $H^{2-s}$ uniformly. However the $L^2$-in-time class improves only to $H^1$.  In the mixed or local cases, no further information is extracted from this computation. We obtain another series of a priori bounds:
\begin{align}
u &\in L^\infty(0,T; H^{2-\sigma} ) \cap L^2(0,T; H^{1}) \label{e:apru2} \\ 
\rho &\in L^\infty(0,T;  H^{1})   \cap  L^2(0,T; H^{\sigma/2 + 1})  \label{e:aprrho2}
\end{align}
where we identify $L^2= H^{0}$.  In addition, we record
\be\label{Xbdh1}
X \in L^\infty(0,T; H^1).
\ee
Notice that in the local/mixed case, the only improvement coming from the $\cH_1$-entropy at the level of the $L^2$-in-time for $\rho$, as well as the boundedness of the $X$ quantity \eqref{Xbdh1}.  The latter point is important in continuing our procedure.
\\

\subsubsection{\bf $\cH_2$-entropy and its consequences} 
Let us denote $Z = \rho^{-1} Y_x$. Then $Z$ satisfies
\begin{equation}
\dD_t Z = - \rho^{-1} (\rho^{-1} h_{xx}(\rho))_x + \rho^{-1}(\rho^{-1} f_x)_x.
\end{equation}
We thus define our next entropy by
\begin{equation}\label{e:H2}
\begin{split}
\cH_2 &=  \frac12\int_\T \rho Z^2 \dx + \int_\T 	\pi_2 \dx \\
\pi_2 & = \frac12 h'(\rho) \frac{\rho_{xx}^2}{\rho^4}.
\end{split}
\end{equation}
Note that 
\[
\ddt \frac12\int_\T \rho Z^2 \dx = - \int_\T ( \rho^{-1} h_{xx})_x Z \dx + \int_\T (\rho^{-1} f_x)_x Z \dx.
\]
The main term we need to eliminate is in fact $ \rho^{-1} h_{xxx} Z$.  The remaining ones coming from $h_{xx}$ and $\rho^{-1}$ end up being bounded by $|\rho_x|^2 + |\rho_x||\rho_{xx}|$. This can be estimated by 
\begin{equation}
\int_\T | Z|  ( |\rho_x|^2 + |\rho_x||\rho_{xx}|  )\dx \leq  \cH_2^{1/2}( |\rho_x|_\infty |\rho_{xx}|_2 + |\rho_x|^2_4).
\end{equation}
Keeping in mind that at this stage the dissipation term will be in $H^{2+\sigma/2}$ (see discussion at the start of section of $\cH_1$-entropy).  Moreover, the density is uniform in $H^1$, we obtain, by interpolation between $H^1$ and $H^{2+\sigma/2}$,
\begin{equation}\label{e:rrr}
|\rho_x|_\infty |\rho_{xx}|_2 + |\rho_x|^2_4 \lesssim \|\rho\|_{H^{2+\sigma/2}}^{\frac{3}{\sigma+2}} + \|\rho\|_{H^{2+\sigma/2}}^{\frac{1}{\sigma+2}}.
\end{equation}
Here, obviously,  $\frac{3}{\sigma+2} \leq 1$ if $\sigma=2$ or $\sigma=s$ in our range of $s$. Hence the term can be hidden in the dissipation,
\begin{equation}\label{e:Zrho}
\int_\T |Z|  ( |\rho_x|^2 + |\rho_x||\rho_{xx}|  )\dx \leq  c(\e) \cH_2 + \e  \|\rho\|_{H^{2+\sigma/2}}^2.
\end{equation}
In the main term $ \rho^{-1} h_{xxx} Z$ the worst part comes when all derivatives fall on the density: 
\[
h_{xxx} = h' \rho_{xxx} + 3 h'' \rho_{xx} \rho_x + h''' \rho_x^3.
\]
Indeed, the term in the middle repeats \eqref{e:Zrho}, while  the last one admits 
\begin{equation}\label{e:Zrho3}
\int_\T |Z|  |\rho_x|^3 \dx \leq \cH_2^{1/2}|\rho_x|^3_6 \leq \cH_2^{1/2} \|\rho\|_{H^{2+\sigma/2}}^{\frac{2}{\sigma+2}} \leq c(\e) \cH_2 + \e  \|\rho\|_{H^{2+\sigma/2}}^2.
\end{equation}
Here and throughout we repeatedly use  interpolation between $H^1$ and $H^{2+\sigma/2}$ for the density terms. Thus the worst part of the original term $ - \int_\T ( \rho^{-1} h_{xx})_x Z \dx$ gets reduced to just 
\begin{equation}\label{e:worst2}
- \int_\T \rho^{-1} h'(\rho) \rho_{xxx} Z \dx.
\end{equation}
As on the $\cH_1$-step  we expect this term to be canceled by a contribution coming from the pressure potential $\pi_2$. Let us examine it next. 

It will be convenient to denote $g(r) = \frac12 h'(r) / r^4$, and simply write $\pi_2 =  g(\rho) \rho_{xx}^2$.  Integrating by parts  we obtain 
\begin{equation}\label{e:derpi2}
\ddt \int_\T \pi_2 \dx = g \rho \rho_{xxx}u_{xx} + g' \rho \rho_{xx} \rho_x u_{xx} - 2 g \rho_{xx} \rho_x u_{xx} - \frac12 g' \rho \rho_{xx}^2 u_x - \frac52 g \rho_{xx}^2 u_x.
\end{equation}
In the course of subsequent computations we will encounter a number of similar terms. They can be sorted into two groups -- local ones involving  $u$ and $X$, and non-local involving operator $\cL^s$.  All terms come with a prefactor of the form $q(\rho)$ for smooth $q$ which can be ignored. The local ones are
\begin{equation}\label{e:listloc}
\rho_{xx}^2 u_x, \quad  \rho_{xx} \rho_x^2 u_x, \quad  \rho_{xx} \rho_x^2 X_x, \quad  \rho_{xx}^2 X_x, 
\end{equation}
\begin{equation}
 \rho_x^2\rho_{xx}^2, \quad  \rho_{xx}\rho_x^4, \quad \rho_{xx}^3
\end{equation}
the non-locals are  $q(\rho)$-multiples of 
\begin{equation}\label{e:listnonloc}
\rho_{xx}^2 \cL^s \rho,  \quad \rho_{xx} \rho_x^2 \cL^s \rho,  \quad \rho_{xx} \rho_x \cL^s \rho_x.
\end{equation}
Substituting $u_x= \rho Y - c_{\nl}\cL^s \rho - c_\loc( \mu(\rho)\rho_x/\rho^2 )_x$ in the first two local terms, we reduce it to the next locals and non-local ones. We now use interpolation and boundedness of $X_x$ in $L^2$ uniformly, to estimate the local terms as follows
\begin{align}
\int_\T	|\rho_{xx}^2 X_x| \dx & \leq |X_x|_2 |\rho_{xx}|_4^2 \leq c_1 \|\rho\|_{H^{2+\sigma/2}}^{\frac{10}{2\sigma + 4}} \leq c_2 + \e \|\rho\|_{H^{2+\sigma/2}}^2 \\
\int_\T| \rho_{xx} \rho_x^2 X_x|\dx & \leq |X_x|_2 |\rho_{xx}|_2 |\rho_x|_\infty^2 \leq c_3 \|\rho\|_{H^{2+\sigma/2}}^{\frac{2}{\sigma + 2}} \|\rho\|_{H^{2+\sigma/2}}^{\frac{2}{\sigma + 2}} = c_3 \|\rho\|_{H^{2+\sigma/2}}^{\frac{4}{\sigma + 2}} \leq c_4 + \e \|\rho\|_{H^{2+\sigma/2}}^2.
\end{align}
The other local terms are
\[
\begin{split}
\left| \int_\T  \rho_{xx}^3 \dx \right|  & \leq    |\rho_{xx}|_\infty \| \rho \|_{\dot{H}^{3}} \| \rho\|_{\dot{H}^{1}} + |\rho_{xx}|_\infty^2 \| \rho\|_{\dot{H}^{1}}^2  \leq c_5 \| \rho \|_{\dot{H}^{3}}^{\frac{7}{4}} + c_6 \| \rho \|_{\dot{H}^{3}}^{\frac{3}{2}} \leq c_7 + \e \|\rho\|_{H^{3}}^2.
\end{split}
\]
which follows after noting that $\| \rho\|_{\dot{H}^{1}}$ is uniformly bounded and $|\rho_{xx}|_\infty \leq c \| \rho \|_{\dot{H}^{3}}^{\frac{3}{4}}$. In the next two terms we simply use H\"older inequality:
\[
\left| \int_\T  \rho_{xx}^2 \rho_x^2  \dx \right|   \leq |\rho_x|_\infty^2 \| \rho \|_{H^{2}}^2 \leq \| \rho \|_{\dot{H}^{3}}^{\frac{1}{2}}  \| \rho \|_{\dot{H}^{3}}=  \| \rho \|_{\dot{H}^{3}}^{\frac{3}{2}} \leq c_7 + \e \|\rho\|_{H^{3}}^2.
\]
\[
\left| \int_\T  \rho_{xx}\rho_x^4 \dx \right|   \leq |\rho_{xx}|_2 |\rho_x|_\infty \| \rho \|_{H^{3}} \leq \| \rho \|_{\dot{H}^{3}}^{\frac{1}{2}}  \| \rho \|_{\dot{H}^{3}}^{\frac{1}{4}}  \| \rho \|_{\dot{H}^{3}} =  \| \rho \|_{\dot{H}^{3}}^{\frac{7}{4}} \leq c_7 + \e \|\rho\|_{H^{3}}^2.
\]

Let us turn to non-local ones. In the first one we symmetrize in the operator  $\cL^s$, which produces increments of the other factors that come with it. Thus, we have 
\[
\begin{split}
\left| \int_\T q(\rho) \rho_{xx}^2 \cL^s \rho \dx \right|  & \leq    |\rho_{xx}|_\infty \| \rho \|_{\dot{H}^{s/2+2}} \| \rho\|_{\dot{H}^{s/2}} + |\rho_{xx}|_\infty^2 \| \rho\|_{\dot{H}^{s/2}}^2 \\
\intertext{and noting that $\| \rho\|_{\dot{H}^{s/2}}$ is uniformly bounded and $|\rho_{xx}|_\infty \leq c \| \rho \|_{\dot{H}^{s/2+2}}^{\frac{3}{s+2}}$,}
& \leq c_5 \| \rho \|_{\dot{H}^{s/2+2}}^{1+ \frac{3}{s+2}} + c_6 \| \rho \|_{\dot{H}^{s/2+2}}^{\frac{6}{s+2}} \leq c_7 + \e \|\rho\|_{H^{2+\sigma/2}}^2.
\end{split}
\]
In the next we simply use H\"older inequality:
\[
\begin{split}
\left| \int_\T q(\rho) \rho_{xx} \rho_x^2 \cL^s \rho \dx \right|  &  \leq |\rho_{xx}|_2 |\rho_x|_\infty^2 \| \rho \|_{H^{s}} \leq \| \rho \|_{\dot{H}^{s/2+2}}^{\frac{2}{s+2}}  \| \rho \|_{\dot{H}^{s/2+2}}^{\frac{2}{s+2}}  \| \rho \|_{\dot{H}^{s/2+2}}^{\frac{2s-2}{s+2}} \\
&=  \| \rho \|_{\dot{H}^{s/2+2}}^{\frac{2s+2}{s+2}} \leq c_7 + \e \|\rho\|_{H^{2+\sigma/2}}^2.
\end{split}
\]
The same strategy applies for the last one:
\[
\begin{split}
\left| \int_\T q(\rho) \rho_{xx} \rho_x \cL^s \rho_x \dx \right|  & \leq |\rho_{xx}|_2 |\rho_x|_\infty \| \rho \|_{H^{s+1}} \leq \| \rho \|_{\dot{H}^{s/2+2}}^{\frac{2}{s+2}}  \| \rho \|_{\dot{H}^{s/2+2}}^{\frac{1}{s+2}}  \| \rho \|_{\dot{H}^{s/2+2}}^{\frac{2s}{s+2}} \\
&=  \| \rho \|_{\dot{H}^{s/2+2}}^{\frac{2s+3}{s+2}} \leq c_7 + \e \|\rho\|_{H^{2+\sigma/2}}^2.
\end{split}
\]
Let us now get back to \eqref{e:derpi2}. The last two terms are obviously of same local type. The two terms in the middle are of the same type too. There we 
replace $u_{xx}$ with
\begin{equation}\label{e:uxx}
u_{xx} = \rho^2 Z + \rho^{ -1} \rho_x u_x + c_{\loc}(\rho^{-1} \rho_x \cL^s \rho - \cL^s \rho_x) +c_{\nl}(q_1 \rho_x^3 + q_2\rho_x\rho_{xx} -q_3 \rho_{xxx})
\end{equation}
where $q$'s are some functions of $\rho$, and most importantly,  $q_3=\mu(\rho)/\rho^2 > 0$.
This results into $ \rho_{xx} \rho_x Z$, already estimated in \eqref{e:Zrho}; and the series of terms
\[
\rho_{xx} \rho_x^2 u_x,\  \rho_{xx} \rho_x^2 \cL^s\rho,\  \rho_{xx} \rho_x \cL^s\rho_x, \  \rho_x^2\rho_{xx}^2,\    \rho_{xx}\rho_x^4, \  \rho_{xx}^3,
\]
all of which have been protocoled above. Finally, in the first and main term in  \eqref{e:derpi2} we use \eqref{e:uxx} to obtain
\[
g \rho \rho_{xxx}u_{xx} = \rho^{-1} h'(\rho) \rho_{xxx} Z  +   g \rho_{xxx}(\rho_x u_x +  \rho_x \cL^s \rho - \rho \cL^s \rho_x) +  g \rho_{xxx}(  q_1 \rho_x^3 +q_2 \rho_x\rho_{xx} - q_3 \rho_{xxx}),
\]
The first term is precisely the one that cancels with \eqref{e:worst2}.  In the rest, we integrate by parts to relieve one derivative from  $\rho_{xxx}$ in all but the final term above, which is strictly dissipative.  In the local terms, integrate by parts to relieve one derivative from  $\rho_{xxx}$ and letting $g_i(r):= g(r) q_i(r)$, we obtain
\begin{align*}
-g \rho_{xxx}\rho_x u_x &= g' \rho_{xx} \rho_x^2 u_x + g \rho_{xx}^2 u_x + g \rho_{xx}\rho_x u_{xx}\\
-g_1 \rho_{xxx} \rho_x^3 &=g_1' \rho_{xx} \rho_x^4+ 3g_1 \rho_{xx}^2 \rho_x^2,\\
-g_2\rho_{xxx} \rho_x\rho_{xx} &= \frac{1}{2} g_2\rho_{xx}^3 + \frac{1}{2}  g_2' \rho_x^2\rho_{xx}^2
\end{align*}
upon integration. 
We obtain (up to the opposite sign)
\begin{multline}\label{e:long}
g' \rho_{xx} \rho_x^2 u_x + g \rho_{xx}^2 u_x + g \rho_{xx}\rho_x u_{xx} + g' \rho_{xx} \rho_x^2 \cL^s \rho + g \rho_{xx}^2 \cL^s \rho + g \rho_{xx} \rho_x \cL^s \rho_x \\
- g' \rho_{xx} \rho_x \rho \cL^s \rho_x - g \rho_{xx} \rho_x  \cL^s \rho_x - g \rho_{xx} \rho \cL^s \rho_{xx}\\
+ g \rho_{xx}\rho_x u_{xx}+ g_1' \rho_{xx} \rho_x^4+ \left(3g_1+ \frac{1}{2}  g_2' \right) \rho_{xx}^2 \rho_x^2+  \frac{1}{2} g_2\rho_{xx}^3 - g_3 \rho_{xxx}^2.
\end{multline}
All of the terms have been included in the lists, except $g \rho_{xx} \rho \cL^s \rho_{xx}$ and $g q_3 \rho_{xxx}^2$ which are dissipative. This is obvious for the local term:
\begin{equation}
\begin{split}
\int_\T  g_3 (\rho) \rho_{xxx}^2  \dx & \geq  c_1 \|\rho\|_{\dot{H}^{3}}^2.
\end{split}
\end{equation}
 As to the non-local term, we perform the same argument -- by symmetrization, and noting that variation of $\rho g(\rho)$ results in a variation of $\rho$, we find
\begin{equation}
\begin{split}
\int_\T \rho g(\rho) \rho_{xx} \cL^s \rho_{xx} \dx & \leq  - c_1 \|\rho\|_{\dot{H}^{2+ s/2}}^2 + |\rho_{xx}|_\infty \|\rho\|_{\dot{H}^{s/2}}  \|\rho\|_{\dot{H}^{2+ s/2}} \\
\intertext{since $H^{s/2}$ norm is uniformly bounded,}
& \leq - c_1 \|\rho\|_{\dot{H}^{2+ s/2}}^2 +  \|\rho\|_{\dot{H}^{2+ s/2}}^{\frac{3}{2+\a} + 1} \leq - c_8\|\rho\|_{H^{2+s/2}}^2  + c_9.
\end{split}
\end{equation}

Collecting the obtained estimates we arrive at 
\[
\ddt \cH_2 \leq - c' \|\rho\|_{H^{2+\sigma/2}}^2  + c'' \cH_2 + c'''.
\]
This proves uniform boundedness of $\cH_2$ and  $\rho\in L^2 H^{2+\sigma/2}$.  Also from the corrector we obtain $\rho \in L^\infty H^2$. Since $Z\in L^2$ this translates into $u_{xx} + \p_x \cL^s \rho+ \rho_{xxx} \in L^2$ uniformly. This puts $u\in L^\infty H^{3-\sigma} \cap L^2 H^2$.  Collectively we obtain
\begin{align}
u &\in L^\infty(0,T; H^{3-\sigma} ) \cap L^2(0,T; H^2), \label{e:apru3} \\ 
\rho &\in L^\infty(0,T;  H^2)   \cap  L^2(0,T; H^{\sigma/2 + 2}).  \label{e:aprrho3}
\end{align}
\\

\vspace{-6mm}
\subsubsection{\bf $\cH_n$-entropy: closing the argument}

Let us note that in the previous calculations the requirements on $s$ relax to just $s>1$. It is now clear that can construct a hierarchy of higher order entropies in the form
\begin{equation}\label{e:Hn}
\begin{split}
\cH_n &=  \frac12\int_\T \rho Z_n^2 \dx + \int_\T 	\pi_n \dx \\
\pi_n  &= \frac12 h'(\rho) \frac{(\p^n_x\rho)^2}{\rho^{2n}},
\end{split}
\end{equation}
where at the core of $Z_n$ is the term $\p_x^n u + c_{\nl}\p_x^{n-1} \cL^s \rho+ c_\loc q(\rho) \p_x^{n+1}  \rho$.  The argument above extends easily with identical steps to this general case. As a result we obtain uniform boundedness of the $n$-th entropy and corresponding $L^2$-integrability of the enstrophy. This puts our solution into the classes 
\begin{align}
u &\in L^\infty(0,T; H^{n+1-\sigma} ) \cap L^2(0,T; H^n) \label{e:aprun} \\ 
\rho &\in L^\infty(0,T;  H^n)   \cap  L^2(0,T; H^{\sigma/2 + n})  \label{e:aprrhon}
\end{align}
We thus have shown that
 \be
\begin{aligned}
&\sup_{T\in [0,T^*)} \Vert \rho\Vert_{L^\infty(0, T;  H^n)}+ \sup_{T\in [0,T^*)}\Vert \rho\Vert_{L^2(0, T; H^{\sigma/2 + n})}+\sup_{T\in [0,T^*)}\Vert u\Vert_{L^\infty(0, T; H^{n+1-\sigma} )}+\sup_{T\in [0,T^*)}\Vert u\Vert_{L^2(0, T; H^{n})}\\
&\qquad\qquad\qquad\qquad\le F\Big(\Vert (\rho_0, u_0)\Vert_{H^n(\T)\times H^{n+1-\sigma}(\T)},  \Vert f\Vert_{L^\infty(0, T^*; C^{n})},  \frac{1}{ \cl }, T^*\Big)<\infty
\end{aligned}
\ee
for $n\ge 2$. 
 Appealing to local existence, established by Prop. \ref{prop:local}, the solution can be extended past $T^*$.

\bibliographystyle{plain}
\bibliography{CDSbiblio}

\end{document}